\documentclass[12pt, twoside, leqno]{article}

\usepackage{amsmath,amsthm}
\usepackage{amssymb}

\usepackage{enumerate}

\usepackage{color}

\newtheorem{thm}{Theorem}[section]
\newtheorem{cor}[thm]{Corollary}
\newtheorem{lem}[thm]{Lemma}
\newtheorem{prop}[thm]{Proposition}

\newtheorem{con}[thm]{Conjecture}

\newtheorem{mainthm}[thm]{Main Theorem}

\theoremstyle{definition}
\newtheorem{defin}[thm]{Definition}
\newtheorem{rem}[thm]{Remark}
\newtheorem{exa}[thm]{Example}
\newtheorem{nota}[thm]{Notation}

\numberwithin{equation}{section}

\newcommand{\cn}{\mathbb{C}^n}
\newcommand{\cj}{\mathbb{C}}
\newcommand{\rj}{\mathbb{R}}
\newcommand{\nj}{\mathbb{N}}

\newcommand{\dist}{\mbox{\rm dist}}

\newcommand{\D}{{\mathbb{D}}}

\newcommand{\K}{{\cal K}}

\begin{document}


\baselineskip=17pt


\title{  On Lagrange polynomials\\ and the rate of  approximation\\ of planar sets by polynomial Julia sets}

\author{ Leokadia Bialas-Ciez and Marta Kosek\\ Institute of Mathematics\\ Faculty of Mathematics and Computer Science\\ Jagiellonian University\\ \L ojasiewicza 6, 30-348 Krak\'ow, POLAND\\
E-mails:  Leokadia.Bialas-Ciez@uj.edu.pl\\ and
 Marta.Kosek@im.uj.edu.pl
\and  and\\
Ma\l gorzata Stawiska\\
Mathematical Reviews\\416 Fourth St., Ann Arbor, Michigan, USA\\
E-mail: stawiska@umich.edu
\\\\
 \textit{To the memory of Professor J\'ozef Siciak}}


\maketitle


\renewcommand{\thefootnote}{}



\renewcommand{\thefootnote}{\arabic{footnote}}
\setcounter{footnote}{0}


 \begin{abstract}

  We revisit the  approximation of  nonempty compact planar sets  by filled-in Julia sets of polynomials developed in \cite{arxiv} and  analyze the rate of approximation.  We use  slightly modified  fundamental Lagrange interpolation polynomials and show that taking   certain classes of nodes with subexponential growth of Lebesgue constants improves the approximation rate. To this end we investigate  properties of some arrays of points in $\mathbb{C}$. In particular we prove subexponential growth of Lebesgue constants for   pseudo Leja sequences with bounded Edrei growth on finite unions of quasiconformal arcs. Finally, for some classes of sets we estimate more precisely the rate of approximation by filled-in Julia sets in Hausdorff and Klimek metrics.

\end{abstract}

\section{Introduction}

Julia sets of complex polynomials have been studied in many aspects (for an introduction to the topic, see e.g. \cite{CG}). Recently an interest arose in  approximation of planar sets by polynomial  Julia sets in the Hausdorff metric.  One possible approach to this problem can be found in \cite[Theorem 3]{nagoya}, where
 the approximating sets are   composite Julia sets, defined by means of families of quadratic polynomials.  Another line of research    was  initiated by K. Lindsey in \cite{L15} and further advances regarding possibilities of such approximation were made  e.g. in  \cite{ivrii1},  \cite{ivrii2}, \cite{pixie} and \cite{bishop}.  The article \cite{arxiv} gave a characterization of the sets which can be approximated arbitrarily well by filled-in Julia sets of polynomials  in such a way that their boundaries are also approximated by Julia sets of these polynomials with the same accuracy.  A part of this characterization was provided by the  following theorem:

\begin{thm}[\mbox{\cite[Theorem 1.2]{arxiv}}]\label{thm:LYtheorem} Let $E \subset \mathbb{C}$ be any nonempty compact set with connected complement. Then for any $\varepsilon > 0$  there exists a polynomial $P$ such that
\[
\chi (E,\mathcal{K}(P)) < \varepsilon, \quad \chi (\partial E,\mathcal{J}(P)) < \varepsilon,
\]
where $\mathcal{K}(P),\ \mathcal{J}(P)$ are respectively the filled-in Julia set and the Julia set of $P$ and $\chi$ denotes the Hausdorff metric.

\end{thm}

 The construction of  filled-in Julia sets such that  Theorem  \ref{thm:LYtheorem} holds was carried out in \cite[Theorem 3.2]{arxiv} and  involved the use of  polynomials appearing in  approximation of the Green function $g_E$ of $E$ in   $\mathbb{C} \setminus E$.  Recall that many  sequences of polynomials can be taken into consideration for the purpose of approximating the Green function associated with  a compact set $E$ of logarithmic capacity ${\rm cap} E >0$ with connected complement.  A theorem due  to L. Kalm\'ar and (independently) J. L. Walsh  (see \cite[Theorem 7.4]{Walsh}  or \cite[Theorem 1.4]{BBCL}) says that, for a sequence $(w_n)$ of polynomials with zeros in $E$ such that ${\rm deg} w_n=n$ for $n\in\{1,2,...\}$, one has   $|w_n(z)|^{1/n} \longrightarrow {\rm cap} E \cdot \exp(g_E(z))$ uniformly on compact
subsets of $\mathbb{C}\setminus E$ if and only if $\|w_n\|_E^{1/n} \longrightarrow {\rm cap} E$ as $n \to \infty$.
Conditions implying these  statements are listed e.g. in   \cite[Theorem 1.5.I]{BBCL}.   A quantitative version of  Theorem \ref{thm:LYtheorem} (\cite[Theorem 4.12]{arxiv}) was also proved.  In this article we will  examine the rate of approximation by Julia sets and provide  examples of  sequences of  polynomials  that guarantee  a better rate of approximation than the one in \cite{arxiv} (see Section \ref{s:filled} below).\\

We will devote a lot of attention to
pseudo Leja points (see Definition \ref{def:psleja} below).  Such points were introduced and studied in \cite{BCC}.  Their significant advantage lies in the fact that they are relatively easy to compute even though they are a generalization of classical Leja points.  When $n$ is large, it is almost impossible to find numerically the Fekete or Leja $n$-tuples (see e.g. \cite{BSV} for discussion of computational complexity). And without  numerical tools these tuples are available only for a few classical examples. In contrast, it is quite easy to find the $n$-th point of a pseudo Leja sequence. It suffices to cover the outer boundary of  an infinite compact set $E$ by pairwise externally tangent disks with a fixed radius (depending on $n$) and pick one point in each of those disks. One of these points has to satisfy the condition from the definition of the $n$-th pseudo Leja point (see Section \ref{s:lagrange}  for more details).

In Section \ref{s:lagrange} we establish a few new results for arrays of distinct points of compact sets, in particular  subexponential growth of Lebesgue constants for certain pseudo Leja sequences. The known
open problem on the growth of Lebesgue constants for general Leja points is related to our research, hence we discuss this question  too, in  quite substantial detail. Our main result gives a quite general equivalence concerning the growth of Lebesgue constants $\Lambda_n(E,\zeta^{(n)})$  (see Definition \ref{def:Lebesgue}):

\begin{mainthm}\label{mthm:Lebesgue}
Let $E$ be a polynomially convex, regular, compact set in $\cj$ and for each $n$ let $\zeta^{(n)}=\{\zeta_0^{(n)},...,\zeta_n^{(n)}\}$  be a $(n+1)$-tuple of distinct points  in $E$. Then
the following conditions are equivalent:
\begin{eqnarray*}\bullet & \lim\limits_{n\rightarrow\infty} \Lambda_n(E,\zeta^{(n)})^{1/n} = 1,\\
\bullet & \lim\limits_{n\rightarrow\infty} \min\limits_{k\in\{0,...,n\}} \left|\prod_{j\in\{0,...,n\}\setminus\{k\}}\left(\zeta_j^{(n)}-\zeta_k^{(n)}\right)\right|^{1/n} = {\rm cap}\,E.\end{eqnarray*}
\end{mainthm}

  While investigating pseudo Leja sequences we also observed two nice separation properties for them, namely Proposition \ref{prop:sep} and Theorem \ref{thm:deltan}, which  can be  useful in further research in  interpolation theory.

All terminology and auxiliary results will be introduced  in subsequent sections. Here we would like to recall the following notion:
\begin{defin} Let $E$ be a  compact subset of $\cj$. The {\it polynomially convex hull} of $E$ is
\[
\widehat{E}=\left\{z \in \mathbb{C}:\; |p(z)| \leq \|p\|_E:=\max_{w \in E}|p(w)| \mbox{ for every polynomial } p\right\}.
\]
The set $E$ is called {\it polynomially convex} if $E =\widehat{E}$.
\end{defin}
It is known (e.g. by Runge's theorem) that in $\mathbb{C}$ the class of polynomially convex sets is the same as the class of compact sets with connected complements.   This class is important  in  interpolation theory.  For the purpose of approximating planar sets by Julia sets it is worth recalling  that filled-in Julia sets are polynomially convex.

 Thanks to Main Theorem \ref{mthm:Lebesgue} we are able to prove Theorem \ref{thm:bounds}, which is a generalization of \cite[Theorem 1]{Siciak}.  Together with Lemma \ref{lem:normL_n} it   establishes the convergence  $$\frac1n\log|L_n| \longrightarrow g_E, \qquad n\rightarrow \infty,$$ where $L_n$ are appropriately chosen fundamental Lagrange interpolation polynomials $L^{(j_n)}(\cdot,\zeta^{(n)})$ (see Notations \ref{nota:lagrange} below),  and gives an estimation of its rate. When the Green function $g_E$ is H\"older continuous, Theorem \ref{thm:bounds}
offers a rate of approximation of $g_E$  which  in certain cases is better than the general one given in \cite{Pritsker11}, namely $O\left(\frac{\log n}n\right)$ rather than $O\left(\frac{\log n}{\sqrt{n}}\right)$. This allows us to improve the rate of approximation of $E$ by filled-in Julia sets. For some special classes of sets we get more precise  estimates of the rate of approximation in Section \ref{s:rates}, namely

\begin{prop} [cf. Corollary \ref{wn:Loja}]
Let $E$ be a compact set
satisfying  the \L ojasiewicz-Siciak condition. Assume also that the Green function $g_E$ is H\"older continuous with exponent $\alpha\in (0,1]$. Then there exist positive constants $C, \kappa$, depending only on the set $E$ and, for each $n\in\nj$, large enough,  a polynomial $P_n$ of degree $n+1$ such that
  \begin{itemize}
    \item $\chi(E, {\cal K}(P_n))\leq C \left(n^{-1}\log(n+1)\right)^\kappa, \quad {\it if} \quad \alpha\in [1/2, 1]$;
    \item $\chi(E, {\cal K}(P_n))\leq C n^{-2\alpha\kappa}, \quad {\it if} \quad \alpha\in (0,1/2)$.
  \end{itemize}
\end{prop}

The Green function plays an important role in our article.  In particular,  we explore approximation in the metric $\Gamma(E,F):=\max(||g_E||_F,\; ||g_F||_E)$, defined for compact, regular, polynomially convex subsets of $\mathbb{C}$.  This metric was introduced in \cite{pams}. In general it is not equivalent to the Hausdorff metric, so the results are of independent interest.

\section{The Green function
 and families of neighbourhoods}\label{s:green}

 \subsection{ The Green function }

\begin{defin}
 Let $D \subset \mathbb{C}$ be an unbounded domain. Consider a function  $g:D \longrightarrow \mathbb{R}$ with the following properties:
\begin{enumerate}
\item $g$ is harmonic and positive in $D$;

\item $g(z)$ tends to $0$ as $z \to \partial D$;

\item $g(z)-\log |z|$ tends to a finite number $\gamma$ as $z \to \infty$.
\end{enumerate}
If such a function exists, it is unique. We call $g$ the {\it Green function} of $D$ with pole at infinity.
\end{defin}

Consider a compact subset $E$ of $\cj$. Let $\widehat{E}$ denote    its polynomially convex hull and let  $D_\infty:=\cj\setminus \widehat{E}$. The set $E$  will be called {\it regular} if the Green function of $D_\infty$ with pole at infinity exists. Slightly abusing the terminology, we will denote this Green function  by $g_E$ and call it the  Green function of $E$.  We extend $g_E$ by 0 on $\widehat{E}$, which yields a continuous function in $\mathbb{C}$.  Note that it follows from this definition that $g_E=g_{\widehat{E}}$.
The number ${\rm cap }E =\exp (-\gamma)$ is called the \textit{logarithmic capacity}  of $E$ (and is the same as its \textit{transfinite diameter}).   For  more  background we refer the reader e.g. to \cite{Ransford}.

Let us fix some notations.

\begin{nota}\label{nota:zbiory}
Let $\varepsilon>0$, $n\in\nj$  and $E\subset \cj$.\\ Then
$E^\varepsilon:=\left\{z\in \cj: \dist\left(z, {E}\right)< \varepsilon\right\}$  is the $\varepsilon$-{\it dilation} of $E$ and
$D_n:=\left\{z\in \cj: \dist\left(z, \widehat{E}\right)\geq 1/n^2\right\}$.\\ If the set $E$ is compact and regular,
then $E_\varepsilon:=\left\{z\in \cj: g_E(z)\leq \varepsilon\right\}$ is the $\varepsilon$-sublevel set of $g_E$ and $\Omega_\varepsilon:=\left\{z\in \cj: g_E(z)>\varepsilon\right\}$.
\end{nota}

Observe  that $D_n=\cj\setminus  \left(\widehat{E}\right)^{1/n^2}$ and that $\Omega_\varepsilon=\cj\setminus E_\varepsilon$. We also have $\bigcup_{n=1}^\infty D_n=D_{\infty}$.

It is obvious that the family $\{E^\varepsilon\}_{\varepsilon >0}$ forms a neigbourhood base of the set $E$ in $\cj$.    It is  less obvious but also true  (see \cite[Corollary 1]{pams}) that for $E$ regular and polynomially convex  the family $\{E_\varepsilon\}_{\varepsilon >0}$ forms such a base too.  The sets $E_\varepsilon$ have the following nice properties:

\begin{prop}\label{prop:sublevel}  Let $E$ be a regular compact subset of $\cj$. Then:
\\ $(i)$  For every $\varepsilon >0$ the set $E_\varepsilon$ is polynomially convex.
\\$(ii)$  For every $\varepsilon >0$ the set $E_\varepsilon$ is regular.
\\$(iii)$  $E_{\varepsilon+\tau}=(E_{\varepsilon})_{\tau}$ for every $\varepsilon, \tau >0$.
\end{prop}

\begin{proof}
All these properties follow from a proposition due to M. Mazurek (first published in \cite[Proposition 5.11]{Mazurek})  saying that $g_{E_\varepsilon}=\max(0,g_E-\varepsilon)$.
\end{proof}

 We will say that $\omega:[0,\infty)\longrightarrow \rj$ is a {\it modulus of continuity} of $g_E$ if $\lim_{\delta\rightarrow 0^+}\omega(\delta)=0$ and $g_E(z)\leq \omega(\delta)$ if $\dist(z, E)<\delta$.

 In a later section we will consider a special modulus of continuity and use the following lemma, which is also  a straightforward consequence of the formula $g_{E_\varepsilon}=\max(0,g_E-\varepsilon)$:

\begin{lem}\label{lem:holder}
Let $E$ be a regular compact subset of $\cj$. If $g_E$ is H\"older continuous and $\varepsilon>0$, then $g_{E_
\varepsilon}$ is also H\"older continuous with the same constant and exponent.
\end{lem}

It is known that for every compact connected set $E$ containing more than two points its Green function $g_E$ is H\"older continuous.

In our paper  polynomial Julia sets play the  major role. The following remark  recalls their definition and some useful properties, including also H\"older continuity of the Green function.

\begin{rem}\label{rem:Julia}
For   a complex polynomial $P$ of degree at least 2, both sets $A_P(\infty)= \{z\in\cj: (P^{\circ n}(z))_{n=1}^\infty \ \text{is unbounded}\}$ and $\K(P):=\mathbb{C}\setminus A_{P}(\infty)$ are non-empty (here  $P^{\circ n}$ denotes the $n$-th iterate of $P$). Moreover,  $\partial \K(P) = \partial A_P(\infty)= \mathcal{J}(P)$, the {\it Julia set} of $P$. The set $\K(P)$, which we will call the \textit{filled-in Julia set} of $P$, is compact, polynomially convex and regular. It is known that the Green function of $\K(P)$ is  H\"older continuous, cf. \cite[Theorem VIII.3.2 and the remarks below it]{CG}.
\end{rem}

We will also use the following function defined for a bounded set $E \subset \mathbb{C}^N$ by J. Siciak (who significantly modified an earlier version introduced in one variable by F. Leja) in \cite{Siciak-extr} and further developed in \cite{Zah} and \cite{Mazurek}:

\begin{defin}\label{defin:extremalfcn} Let $E \subset \mathbb{C}$ be bounded.  We define
\[
\Phi_E(z) =\sup_{n \geq 1} \Phi_{n,E}(z), \quad z \in \mathbb{C},
\]
with
\[
\Phi_{n,E}(z)=\sup \{|p(z)|^{1/n}: p \in \mathcal{P}_n^*\},\quad z \in \mathbb{C},\  n \in \mathbb{N},
\]
 where  $\mathcal{P}_n^*$ denotes the space of all complex polynomials $p$ of degree at most $n$ such that  $||p||_E\leq 1$. 
\end{defin}

It is known that $\log \Phi_E = g_E$ for $E$ regular (see e.g. \cite[Theorem 7.3]{Mazurek}).
  This implies the following simple and well known fact, which will be useful later:
\begin{lem}\label{lemma:continuity} Let $E$ be a polynomially convex regular compact subset  of $\cj$, let $n$ be a natural number and let $\alpha_1,...,\alpha_n$ be distinct points in $E$. Then
\[
\lim_{z \to \infty}\log\frac{\Phi_E(z){\rm cap}E}{(\prod_{k=1}^n |z-\alpha_k|)^{1/n}}\,  =\lim_{z \to \infty}\log\frac{\Phi_E(z){\rm cap}E}{|z|}=0.
\]

\end{lem}

 \subsection{Families of neighbourhoods of planar sets}

 In our approximation we will use two metrics  on some families of subsets of $\cj$. The first one is the Hausdorff metric.
Let us recall two equivalent definitions. If $X, Y$ are bounded and nonempty we can put
\begin{eqnarray*}
\chi(X,Y)&:=&\max(||\dist(\cdot, X)||_Y,\; ||\dist(\cdot, Y)||_X)\\
&=&\inf\{ \varepsilon >0: X \subset Y^\varepsilon, \; Y \subset  X^\varepsilon\}.
\end{eqnarray*}
This $\chi$ is a pseudometric and if we restrict it to nonempty compact sets, it is a metric.

In \cite{pams}, M. Klimek defined  another  function for pairs of compact regular sets using their Green functions. Namely,
$$
\Gamma(E,F):=\max(||g_E||_F,\; ||g_F||_E)
$$
for $E, F$  regular. Observe that we  also have
$$
\Gamma(E,F)=\inf\{\varepsilon>0: E\subset F_\varepsilon,\; F\subset E_\varepsilon\}.
$$
 This $\Gamma$ is  a pseudometric  which is a metric under restriction to
the space of regular compact polynomially convex sets.  Since $\Gamma(E,F)=||g_E-g_F||_\cj$ (see \cite{pams}), the convergence of a sequence of regular sets in $\Gamma$ is equivalent to the uniform convergence of their Green functions.\\

Let $E, F$ be regular compact subsets of $\cj$. Then $\Gamma(E,F)=\Gamma(\widehat{E},\widehat{F})$. However, in general $\chi(E,F)\neq \chi(\widehat{E},\widehat{F})$. Consider for instance $E:=\{z\in \cj: |z|=1\}$ and $F:=\{z\in \cj: |z|\leq 1\}$. We have then $\chi(E,F)=1$ and $\chi(\widehat{E},\widehat{F})=0$, since $\widehat{E}=\widehat{F}=F$.

 In this article we are interested in approximation of compact sets by filled-in Julia sets, which are polynomially convex. It will be done from above (with respect to inclusion of sets), i.e.,   in such a way that the approximating sets are  supersets of the approximated one (unlike  e.g. in the case considered in \cite{ivrii1}).

 Before we start any computations,
let us look  at a more general situation regarding a choice of a family of sets approximating a given set. First, for a nonempty compact set $E$, we can  consider the most natural family of compact sets approximating this set from above in the Hausdorff metric, namely the descending family of closures of dilations of $E$, i.e. $\{\overline{E^\varepsilon}\}_{\varepsilon>0}$. We have $\chi(E, \overline{E^\varepsilon})=\varepsilon$.

\begin{prop}\label{prop:otoczki}  Let $E$ be a compact subset of $\cj$. Then:

$(i)$ For every $\varepsilon>0$ the set $\overline{E^\varepsilon}$ is regular.

$(ii)$ If $E$ is regular, then the family $\{\overline{E^\varepsilon}\}_{\varepsilon>0}$  approximates $E$ in the $\Gamma$ metric.
\end{prop}

\begin{proof}
 The statement $(i)$ is a part of  \cite[Corollary 5.1.5]{Kksiazka}. To prove
$(ii)$, note that, since $\bigcap_{\varepsilon >0} \overline{E^\varepsilon}=E$, we have
$$\lim_{\varepsilon\rightarrow 0} \Gamma(\overline{E^\varepsilon},E) = \lim_{\varepsilon\rightarrow 0}||g_E||_{\overline{E^\varepsilon}}=0.$$
\end{proof}

\begin{rem} Note that $\overline{E^\varepsilon}$ may happen not to be  polynomially convex  even if $E$ is so.  For example, $E=\{e^{2\pi i t}: t\in[0; 0.9]\}$ is polynomially convex and $\overline{E^\varepsilon}$ is not  e.g. if $\varepsilon\in(0.05; 0.4)$.
\end{rem}

Before we look at another family of sets let us note here that if $A$ is bounded    but not necessarily closed, then we can take its polynomially convex hull to be   $\widehat{A}=\widehat{\overline{A}}$.

Consider now the polynomially convex hulls of the dilations, namely  $\left\{\widehat{{E^\varepsilon}}\right\}_{\varepsilon>0}$. This is once again a descending family and we have (cf. \cite[p. 376]{nivoche})
$$\bigcap_{\varepsilon >0} \widehat{{E^\varepsilon}}=\widehat E.$$
Thus if $E$ is polynomially convex, then
\begin{equation}\label{convdilation}
\lim_{\varepsilon \rightarrow 0} \chi\left(E, \widehat{{E^\varepsilon}}\right)=0.
\end{equation}
Observe however that $\chi\left(E, \widehat{{E^\varepsilon}}\right)\geq \varepsilon$, e.g. if $E=\{e^{2\pi i t}: t\in[0; 0.9]\}$ we have $\chi\left(E, \widehat{{E^{0.1}}}\right)=0.9>0.1$. Since $g_F\equiv g_{\widehat F}$ for any regular set $F$,  it follows from Proposition \ref{prop:otoczki} that the family $\left\{\widehat{{E^\varepsilon}}\right\}_{\varepsilon>0}$ approximates $E$ also in the metric $\Gamma$  if $E$ is regular.

Note that if $E$ is regular and polynomially convex, one can also take the sublevel sets of the Green function $g_E$, namely $\{E_\varepsilon\}_{\varepsilon>0}$, as an approximating family. This is once again a descending family of polynomially convex sets (as mentioned before) and it approximates $E$ in both metrics; however, in general we cannot estimate the distance $\chi(E, E_\varepsilon)$ even though we know that $\Gamma(E,E_\varepsilon)=\varepsilon$.  Estimates in some special cases will be given in a later section.

    In our approximation process it will be important to consider neighbourhoods of $E$ with   boundaries that are finite unions of quasiconformal arcs. We will give a precise definition of a quasiconformal arc in Section \ref{s:lagrange}. Here we recall that a sufficient condition for an arc   to be quasiconformal is to be of class $\mathcal{C}^2$.  Neighbourhoods with such boundaries are supplied by the family $\{E_\varepsilon\}$, thanks to the following proposition  (see \cite[Section 2.1, Theorem and subsequent discussion on p. 19]{Du}):

\begin{prop}[\mbox{cf. \cite[Section 2.1]{Du}}] \label{prop:C2arcs} The level set of a non-constant harmonic function $u$ through a critical point consists locally of two or more analytic arcs intersecting with equal angles at $z_0$. If $z_0$ is a regular point of $u$, then the level set of $u$ is locally a single analytic arc passing through $z_0$.
\end{prop}

In fact,   for almost all $\varepsilon > 0$ each connected component of the level set $\{z\in \cj: g_E(z)= \varepsilon\}$   is  a simple closed real analytic curve in $\mathbb{C}\setminus E$ (cf.  \cite[Exercise XV.7.8]{Gamelin}). Smoothness of boundaries of sets in the family $\{E^\varepsilon\}$ is a much more sophisticated issue and will not be discussed here.

\section{Lagrange interpolation polynomials}\label{s:lagrange}

 Before we start dealing specifically with  Lagrange polynomials, let us recall a useful result concerning polynomials in general. Here and throughout the paper we do not take into consideration the polynomial which is identically zero.

Let $E$ be a compact planar set. We will use the following notation:
\begin{equation}\label{Mn}
M_n=M_n(E):=\sup \frac{\|p'\|_E}{\|p\|_E},
\end{equation}
where the supremum is taken over all polynomials $p$ of degree at most $n$ with complex coefficients. In other words, $M_n$ is the norm of the differentiation operator $p\mapsto p'$ acting between two finite dimensional spaces of polynomials with the norm $\|\cdot \|_E$. Sometimes $M_n$ is said to be the ($n$-th) \textit{Markov constant}  because of the following obvious consequence of formula (\ref{Mn}),
\begin{equation}\label{Markov}
\|p'\|_E \le M_n \, {\|p\|_E} \ \ \textrm{for all polynomials of degree not exceeding} \ n,
\end{equation}
which is called a \textit{Markov-type inequality.}
For any regular set the sequence of Markov constants is of subexponential growth, i.e., the following  is true:

\begin{prop}[\mbox{\cite[Theorem 1]{To95}}]\label{prop:Totik} If $E$ is a compact regular set in $\mathbb{C}$ then \[\lim_{n\rightarrow \infty} \sqrt[n]{M_n(E)}=1.\]
\end{prop}

Let us fix now another list of notations.
\begin{nota} \label{nota:lagrange} Let $n\in \nj$ and $E$ be a compact subset of $\cj$.
 Consider an array $A$ of distinct points in $E$, i.e. $A=\{\zeta^{(n)}: n\geq 1\}$ where $\zeta^{(n)}=\{\zeta_0^{(n)}
,...,\zeta_n^{(n)}
\}$, $n\in\nj$,  is an $(n+1)$-tuple of distinct points in $E$. We put
\begin{itemize}
\item $\Theta(E):=(2+{\rm diam}E)^{1/3}$;
\item $w_n(z)=w_n(z;\zeta^{(n-1)})=(z-\zeta_0^{(n-1)})...(z-\zeta_{n-1}^{(n-1)})$;
 \item $V(\zeta^{(n)})=\prod_{0\leq j <k \leq n} \left|\zeta_j^{(n)} -\zeta_k^{(n)}\right|$.
\end{itemize}
Moreover, if $j\in\{0,1,...,n\}$, we put
\begin{itemize}
\item $\Delta^{(j)}(\zeta^{(n)}):=\prod_{k\in\{0,...,n\}\setminus\{j\}}
\left(\zeta_j^{(n)}-\zeta_k^{(n)}\right)$;
\item  $L^{(j)}(z,\zeta^{(n)}):=\left(\Delta^{(j)}(\zeta^{(n)})\right)^{-1}\cdot\prod_{k\in\{0,...,n\}\setminus\{j\}} \left(z-\zeta_k^{(n)}\right) .$
\end{itemize}
\end{nota}
\noindent $L^{(j)}(\cdot,\zeta^{(n)})$ is the $j$-th {\it fundamental Lagrange interpolation polynomial with nodes} $\zeta^{(n)}$.  Note that each $L^{(j)}(\cdot, \zeta^{(n)})$ is a polynomial of degree $n$.

To get polynomials whose Julia sets approximate a given regular set $E$ we will modify some polynomials occurring in approximation of the Green function $g_E$.  For our purposes we will consider fundamental Lagrange interpolation polynomials with  nodes that guarantee good convergence properties, in particular Fekete points and pseudo Leja points.

In the following subsections we will first define the important notion of the Lebesgue constants, then list some facts considering Fekete extremal points, investigate the pseudo Leja sequences and finally prove some general convergence results.

\subsection{Subexponential growth of the Lebesgue constants}

\begin{defin}\label{def:Lebesgue}
Let $E$ be a compact planar set and $A=\{\zeta^{(n)}: n\geq 1\}$
be an array of distinct points in $E$. First, we define the ($n$-th) {\it Lebesgue function}, i.e.
\[ \Lambda_n (z) = \Lambda_n (z, \zeta^{(n)}) := \sum_{j=0}^n |L^{(j)}(z,\zeta^{(n)})|, \qquad z\in \mathbb{C}.\]
 The $n$-th {\it Lebesgue constant} is defined by the formula
\[ \Lambda_n (E) = \Lambda_n(E, \zeta^{(n)}) := \max_{z\in E} \Lambda_n(z,\zeta^{(n)}) = \|\Lambda_n\|_E.\]
\end{defin}

These constants are of outstanding importance in  interpolation theory, see e.g. \cite{BBCL}, \cite{TT}.

By classical interpolation theorems  (see e.g. \cite{BBCL}), if $E$ is a polynomially convex regular compact set and the sequence of Lebesgue constants $(\Lambda_n(E, \zeta^{(n)}))_{n=1}^\infty$ is of subexponential growth, i.e.
\[ \lim_{n\rightarrow\infty} \Lambda_n(E,\zeta^{(n)})^{1/n} = 1\]
then e.g.:
\begin{itemize}
\item the discrete measures  $\mu_n =(1/n) \sum_{j=0}^{n-1}\delta_{\zeta^{(n)}_j}$ converge in the weak* topology to $\mu_E$, where $\delta_{\zeta}$ denotes the Dirac measure supported in $\{\zeta\}$ and $\mu_E$ is the equilibrium measure of $E$;
\item $\lim\limits_{n\rightarrow\infty} |w_n(z)|^{1/n} = {\rm cap}\,E\,\cdot \exp(g_E(z))$ \ for any \ $z\in \mathbb{C}\setminus \widehat{E}$;
\item $\lim\limits_{n\rightarrow\infty} \|w_n\|_E^{1/n} = {\rm cap}\,E$;
\item for any function $f$ holomorphic in a neighbourhood of $E$ 
    its interpola\-tion polynomials with nodes $\zeta^{(n)}$ converge uniformly~on~$E$~to~$f$ (and the convergence is  geometrically fast, i.e. if $p_n$ is the $n$-th interpolating polynomial, then $\|f-p_n\|_E = O(k^n)$ where $k=k(f)<1$).
\end{itemize}

Taking into account the above consequences, we can see that a crucial problem for a given array $A=\{\zeta^{(n)}: n\geq 1\}$
of distinct points in a compact set of $E$ is whether
\begin{equation}\label{sublambda} \lim_{n\rightarrow\infty} \Lambda_n(E,\zeta^{(n)})^{1/n} = 1.\end{equation}

\subsection{Extremal Fekete points}

Consider a regular compact planar set $E$. Let $\eta^{(n)}$ be a Fekete (extremal) $(n+1)$-tuple for $E$  (i.e., an $(n+1)$-tuple maximizing  the absolute value of the Vandermonde determinant, $V(\zeta^{(n)})$, in $E$). We order the points of $\eta^{(n)}$ so that
\begin{equation}\label{order}
|\Delta^{(0)}(\eta^{(n)})| \leq \min_{j\in\{1,...,n\}}|\Delta^{(j)}(\eta^{(n)})|.
\end{equation}

 It is easy to see that
\begin{equation}\label{zgory}
|L^{(j)}(z,\eta^{(n)})|\leq 1,\qquad z\in E, \; j\in\{0,1,...,n\}.
\end{equation}
Furthermore, let $\omega$ be a modulus of continuity of $g_E$. By \cite[Theorem 1]{Siciak}, if $z\in D_n$, then
\begin{eqnarray}\label{FeketeGreenbound}
\qquad 0 \leq g_E(z)-\log \sqrt[n]{|L^{(0)}(z, \eta^{(n)})|}\leq \frac3n\log[(n+1)\Theta(E)]+\omega\left(\frac1{n^2}\right).
\end{eqnarray}
 (Recall that $\Theta(E)$ was defined in Notations \ref{nota:lagrange}).

For extremal Fekete points it is well known and easy to
prove that the Lebesgue constants are of subexponential growth.

\subsection{Pseudo Leja sequences}

Let $E$ be a compact planar set. We define some sequences of extremal points in $E$ associated with the name of F. Leja, who introduced the classical version of such points.

\begin{defin}[\mbox{\cite[Definition 1]{BCC}}]\label{def:psleja} Let $(C_n)_{n=1}^\infty$ be a sequence of real numbers such that
\begin{equation} \label{Edrei}  \qquad
\lim_{n \to \infty} (C_n)^{1/n}=1.
\end{equation}
A sequence of points $a_n \in E, \  n\in\{0,1,2,...\}$, is a  \textit{pseudo Leja sequence in} $E$ \textit{of Edrei growth} $(C_n)_{n=1}^\infty$ if
\begin{equation} \label{pseudoL}
C_n |w_n(a_n)| \geq ||w_n||_E \ \ \ \textrm{for  all} \ n\in\{1,2,...\}.
\end{equation}
The sequence $(a_n)_{n=0}^\infty$ is sometimes simply called a pseudo Leja sequence if it satisfies inequality (\ref{pseudoL}) for some $(C_n)_{n=1}^\infty$ such that condition (\ref{Edrei}) holds. In the case when $C_n=1$ for all $n$, the sequence $(a_n)_{n=0}^\infty$ satisfying (\ref{pseudoL}) is a {\it classical Leja sequence}. We will be mainly concerned with pseudo Leja sequences of {\it bounded Edrei growth}, i.e., the case when $C:=\sup C_n<\infty$.
\end{defin}

Observe that points of a pseudo Leja sequence do not need to belong to the boundary of the set $E$. However,  we will consider only  pseudo Leja sequences that are contained in the outer boundary of $E$, which we can do without loss of generality and in a natural analogy with other classical interpolation nodes (e.g. Fej\'er or Fekete points).

As described in the Introduction, points of pseudo Leja sequences are relatively easy to calculate for any infinite compact
set. In order to find $a_n$  it is sufficient to take an arbitrary point from each of pairwise externally tangent disks of radius $r_n$ covering the outer boundary of $E$ and to verify which of them satisfies the inequality (\ref{pseudoL}). This procedure is justified by \cite[Theorem 4]{BCC} for so-called Markov sets. More generally, an analogous procedure for  infinite compact
sets can be proved in the same manner as in \cite{BCC}  (with $M_n$ in the following proposition  defined as in (\ref{Mn})).

\begin{prop}[\mbox{cf. \cite[Theorem 4]{BCC}}]  If $E\subset \mathbb{C}$ is an infinite compact  set and $(a_n)$ is a pseudo Leja sequence in $E$ of Edrei growth $(C_n)$, then
there exists a closed disk $\D$ of radius $\frac1{M_n}\log(2-\frac1{C_n})$ such that $\D\cap E\neq\emptyset$ and every point $z$ of $\D$ satisfies​ the inequality $C_n|w_n(z)| \ge \|w_n\|_E$. \end{prop}

This proposition says that we can take $r_n=\frac1{M_n}\log(2-\frac1{C_n})$ in the procedure described above.
It is worth noting that in view of \cite[Theorem 1]{E}, for any continuum  $E\subset \mathbb{C}$ we have $M_n(E)\le \frac{2^{1/n-1} \,n^2}{ \textrm{cap}\,E}$.\\

Let us now discuss the Lebesgue constants for pseudo Leja sequences.
A longstanding open question is to verify Condition (\ref{sublambda})  for the classical Leja sequences, see e.g. \cite[Open question (5.1)]{BBCL}.
Recently the  interest in this conjecture has been increasing, see e.g. \cite{CP1}, \cite{CP2}, \cite{TT}, \cite{A}.  We know that the answer is affirmative in the case of the unit disk, thanks to the work done in   \cite{CP1}, \cite{CP2}, \cite{C}, \cite{O} and \cite{Ir}. An interesting approach to obtaining an answer to this question in a general case is given by R. Taylor and V. Totik in  \cite{TT} and V. Andrievskii in \cite{A}.

In order to present the last of the results mentioned, we
need the following notion (cf. \cite{A}, \cite{LVbook}):
\begin{defin}
A bounded arc $L \subset \mathbb{C}$ is {\it quasiconformal}  if and only if for any two points
$z, \zeta \in L$ we have diam$ (L(z, \zeta )) \le  c|z-\zeta|$,
where $c \ge 1$ is a constant and $L(z, \zeta)$ is a subarc of $L$ between these points.
\end{defin}

\begin{thm} [\mbox{\cite[Theorem 1.2]{TT}}, cf. \cite{A}]\label{th:TT} If the outer boundary of a compact infinite set $E$ is a finite union of quasiconformal  arcs, then the Lebesgue constants for a classical Leja sequence are subexponential.
\end{thm}

Recall that the outer boundary of a compact set $E$ is $\partial \widehat{E}=\partial D_\infty$. An inspection of the proofs in \cite{TT} and \cite{A} reveals that the above statement can be generalized to pseudo Leja sequences in $\partial \widehat{E}$ of bounded Edrei growth. In order to prove this, we show the following separation result:

\begin{prop} \label{prop:sep}
Let $E$ be a regular compact set in the complex plane and $(a_n)_{n=0}^\infty$ be a pseudo Leja sequence in $E$ of Edrei growth $(C_n)_{n=1}^\infty$. Then for every $j\in \{1,2,...,n-1\}$ we have the following estimate
\begin{equation*}  |a_j-a_n| \ge \frac1{2e\,C_n} \max\{\varrho_n(a_j), \varrho_n(a_n)\},\end{equation*}
where $\varrho_n(z):=\dist(z,\{w \, : \, g_E(w)=\log(1+1/n)\})$.
\end{prop}

\begin{proof}
  We will first show that
  \begin{equation} \label{pom3} |a_j-a_n| \ge \frac1{2e\,C_n}\, \varrho_n(a_j).\end{equation}
  Put $r:= \frac12 \varrho_n(a_j)$. If $ |a_j-a_n| \ge r$ then the proof of (\ref{pom3}) is completed. In the opposite case, by the definition of pseudo Leja sequence, we have
 \begin{eqnarray*}
 \|w_n\|_E\le C_n |w_n(a_n)| &\le& C_n \!\!\int_{a_j}^{a_n} |w_n'(z)| \, |dz| + C_n |w_n(a_j)|\\& = & C_n \!\!\int_{a_j}^{a_n} |w_n'(z)| \, |dz|.
 \end{eqnarray*}
  Cauchy's formula leads us to $|w_n'(z)|\le \frac1r \|w_n\|_{\D(z,r)}$, where $\D(z,r)$ is the closed disk centered at $z$ with the radius $r$. Observe that \[ \D(z,r)\subset \D(a_j,2r)\subset\{g_E\le \log(1+1/n)\}\]
  and by the properties of the extremal function $\Phi_E=\exp g_E$ (see Definition \ref{defin:extremalfcn}), we have \[ \|w_n\|_{\D(z,r)} \le \|w_n\|_E \, \exp(n\, \|g_E\|_{\D(z,r)}) \le \|w_n\|_E \, (1+1/n)^n \le e\, \|w_n\|_E.\]
  Consequently,
  \[\|w_n\|_E \le \frac{C_n\, e}r \, |a_j-a_n| \, \|w_n\|_E = \frac{2e\,C_n}{\varrho_n(a_j)} \, |a_j-a_n| \, \|w_n\|_E\]
  and we obtain (\ref{pom3}). In the same manner we get the inequality
  \[ |a_j-a_n| \ge \frac1{2e\,C_n}\, \varrho_n(a_n)\]
  and this completes the proof.
\end{proof}

\begin{thm} \label{cor:Leb_dla_Leja} If the outer boundary of a compact set $E$ is a finite union of  quasiconformal arcs  and $(a_n)_{n=0}^\infty$ in $\partial \widehat{E}$ is a pseudo Leja sequence of bounded Edrei growth, then
\begin{equation}\label{Leb_dla_Leja}
  \lim\limits_{n\rightarrow\infty} \Lambda_n(E,a^{(n)})^{1/n} = 1,
\end{equation}
where $a^{(n)}=\{a_0,...,a_n\}$.
\end{thm}

\begin{proof}
Using the reasoning of \cite{TT} and \cite{A}, one can observe that it is sufficient to prove the above statement for the set $E$ which is a bounded quasiconformal arc. In this case consider two conditions:
\begin{enumerate}
  \item[(1)] the discrete measures  $\mu_n =\frac1n \sum_{i=0}^{n-1}\delta_{a_i}$ converge in the weak* topology to the equilibrium measure $\mu_E$,

  \item[(2)] a sequence of points $(a_n)$ in the arc $E$ satisfies the separation inequality
\[ |a_j-a_n| \ge c\, \max\{\varrho_n(a_j), \varrho_n(a_n)\}\]
with a positive constant $c$ independent of $j$ and $n$.
\end{enumerate}
\noindent V. Andrievskii has shown in \cite{A} that $(1)$ and $(2)$ together imply (\ref{Leb_dla_Leja}).

In our case the two above statements are true. Indeed, condition (1) holds for any pseudo Leja sequence, see \cite[Theorem 1]{BCC}, and condition  (2) is true by Proposition \ref{prop:sep} for any pseudo Leja sequence of bounded Edrei growth.
\end{proof}

The proof of the main result in \cite{A}  cannot be easily adapted
to the case of general pseudo Leja sequences without the assumption
that the Edrei growth is bounded. Different methods were used to
prove (in \cite{CP1} and \cite{Ir}) that (\ref{Leb_dla_Leja}) holds for pseudo Leja sequences
arising from Leja sequences on a circle via the Riemann map (by \cite[Theorem 6]{BCC}, such sequences have at most polynomial Edrei growth). But
the question concerning condition (\ref{Leb_dla_Leja}) for an arbitrary pseudo Leja
sequence (even on a smooth curve) is still open. It seems plausible
that for every compact regular set in the complex plane the Lebesgue
constants for any pseudo Leja sequence with bounded Edrei grow
subexponentially.

\begin{con} \label{conjLeb} If $E$ is a regular compact set in $\mathbb{C}$ then
condition $(\ref{Leb_dla_Leja})$ holds for any pseudo Leja sequence of bounded Edrei growth in $\partial \widehat E$.
\end{con}

 Finally, let us propose another nice separation property of pseudo Leja sequences, which we will not use here but which we find interesting for its own reason.
Given a pseudo Leja sequence $(a_n)_{n=0}^\infty$ in $E$, we will consider
$$\sigma_n :=\min_{j\in\{0,...,n-1\}}|a_n-a_j|, \quad n\in\nj.$$ The following result says that $(\sigma_n)$ for a pseudo Leja sequence $(a_n)$ is of subexponential growth.

\begin{thm}\label{thm:deltan} If $E$ is a regular compact set in $\mathbb{C}$ and $(a_n)$ is a pseudo Leja sequence in E  then
 $$
 \lim_{n\rightarrow \infty} \sqrt[n]{\sigma_n}=1.
 $$
 Moreover,
 \begin{equation}\label{separacja}
   \sigma_n \ge \frac{\log(1+1/C_n)}{M_n} \ge \frac1{2\,C_n\, M_n}.
 \end{equation}
\end{thm}

\begin{proof}
Let $(a_n)$ be a pseudo Leja sequence in $E$ with Edrei growth $(C_n)$. Fix a positive integer  $n$ and $j\in \{0,...,n-1\}$. As in the proof of \cite[Theorem 3]{BCC}, it follows from inequality (\ref{pseudoL})  and Taylor's formula that
\[ \frac{\|w_n\|_E}{C_n} \le |w_n(a_n)| \le \sum_{l=0}^n \frac1{l!} |w_n^{(l)}(a_j)| \, |a_j-a_n|^l = \sum_{l=1}^n \frac1{l!} |w_n^{(l)}(a_j)| \, |a_j-a_n|^l.\]
By iterating inequality (\ref{Markov}), we obtain estimates for all derivatives of $w_n$ and thus
\[ \frac{\|w_n\|_E}{C_n} \le \sum_{l=1}^n \frac1{l!} M_n^l \|w_n\|_E \, |a_j-a_n|^l \le \left(e^{M_n|a_j-a_n|}-1\right) \|w_n\|_E.\]
By the definition of $\sigma_n$, we obtain
\begin{equation}\label{pom1}
  \sigma_n \ge \frac{\log(1+1/C_n)}{M_n}.
\end{equation}
Of course, $\sigma_n$ is estimated from above by the diameter of the set $E$ and therefore, it is sufficient to prove that the $n$-th root of the right hand side of this inequality tends to 1. To this end, we will use the following well known estimates
\[ \frac{x}{1+x} \le \log (1+x) \le x \qquad \ \textrm{for} \quad x\ge 0\]
which imply that
\begin{equation}\label{pom2}
  \left[\frac1{(1+C_n)M_n}\right]^{1/n} \le \left[\frac{\log(1+1/C_n)}{M_n}\right]^{1/n} \le \left[\frac1{C_n\, M_n}\right]^{1/n}.
\end{equation}
The lower and the upper bound in this formula tend to 1 in view of (\ref{Edrei}) and Proposition \ref{prop:Totik}.

Inequalities (\ref{separacja}) are obvious consequences of (\ref{pom1}) and (\ref{pom2}).
\end{proof}

\subsection{Some general convergence results}

We will now consider  an array $A$ of distinct points of a regular compact set $E$. Recall Notations \ref{nota:lagrange} and Condition (\ref{sublambda}). The following theorem is a reformulation of our Main Theorem \ref{mthm:Lebesgue}.

\begin{thm}\label{tw:rownowaznosc}
Let $E$ be a polynomially convex  regular  compact set in $\cj$ and $A=\{\zeta^{(n)}: n\geq 1\}$ an array  of distinct points  in $E$. Then $(\ref{sublambda})$ is equivalent to
\begin{eqnarray}
\label{Deltamin}
\lim\limits_{n\rightarrow\infty} \min\limits_{k\in\{0,...,n\}} |\Delta^{(k)}(\zeta^{(n)})|^{1/n} &= &{\rm cap}\,E.\end{eqnarray}
\end{thm}

In what follows, we will write
$$\Delta_k^{(n)}:=\Delta^{(k)}(\zeta^{(n)}), \quad k\in\{0,1,...,n\},$$
for simplicity.

\begin{proof}

\noindent
 $(\ref{sublambda})\Longrightarrow (\ref{Deltamin})$.

Fix $j\in \{0,...,n\}$ such that $|\Delta_j^{(n)}|=\min_k|\Delta_k^{(n)}|$.  Observe that
\begin{eqnarray*} \Lambda_n(E) = \max_{z\in E} \sum_{l=0}^n \frac{|w_{n+1}(z)|}{|z-\zeta_l^{(n)}|\, |\Delta_l^{(n)}|} &\ge& \frac1{{\rm diam}\, E} \max_{z\in E} \frac{|w_{n+1}(z)|}{ |\Delta_j^{(n)}|}\\& = &
\frac1{{\rm diam}\, E}  \frac{\|w_{n+1}\|_E}{|\Delta_j^{(n)}|}. \end{eqnarray*}
Since cap$\,E$ is equal to the Chebyshev constant, which is not greater than $\|p\|_E^{1/n}$ for all monic polynomials $p$ of degree $n$, we have
\[  |\Delta_j^{(n)}| \ge \frac{({\rm cap}\,E)^n}{\Lambda_n(E)\,{\rm diam}\, E}.\]
Hence \[ \frac{{\rm cap}\,E}{(\Lambda_n(E))^{1/n}\,({\rm diam}\, E)^{1/n}} \le \min\limits_{k\in\{0,...,n\}} |\Delta_k^{(n)}|^{1/n}.\]
By the assumption, the left hand side of this inequality tends to cap$\,E$ as $n \to \infty$.

On the other hand,  it follows from (\ref{Markov}) that
\[ |\Delta_k^{(n)}| = |w_{n+1}'(\zeta_k^{(n)})| \le M_{n+1}\|w_{n+1}\|_E \ \ \ \textrm{for any} \ k\in \{0,...,n\}.\]
  In view of     \cite[Theorem 1.4 and Theorem 1.5]{BBCL},  the statement   $(\ref{sublambda})$  implies that $\lim_{n \to \infty} ||w_n||_E^{1/n}={\rm cap }\,E$. This fact and Proposition \ref{prop:Totik} yield
\begin{equation*}
   \min\limits_{k\in\{0,...,n\}} |\Delta_k^{(n)}|^{1/n} \le M_{n+1}^{1/n}\|w_{n+1}\|_E^{1/n} \longrightarrow \textrm{cap}\, E \ \ \ {\rm as} \ n\rightarrow \infty
\end{equation*}
and this finishes the proof of the first implication.\\

\noindent $(\ref{Deltamin})\Longrightarrow (\ref{sublambda})$

Observe first that
\begin{eqnarray*}\Lambda_n(E) = \max_{z\in E} \sum_{k=0}^n \frac{|w_{n+1}(z)|}{|z-\zeta_k^{(n)}|\, |\Delta_k^{(n)}|} &\le& \max_{z\in E^{1/M_{n+1}}} \sum_{k=0}^n \frac{|w_{n+1}(z)|}{|z-\zeta_k^{(n)}|\, |\Delta_k^{(n)}|}\\ & = &\max_{z\in \partial\,\left(E^{1/M_{n+1}}\right)} \sum_{k=0}^n \frac{|w_{n+1}(z)|}{|z-\zeta_k^{(n)}|\, |\Delta_k^{(n)}|} ,\end{eqnarray*}
where $E^\varepsilon$ is the dilation defined in Notation \ref{nota:zbiory}.  The second equality follows from the maximum principle for subharmonic functions. If $z\in \partial (E^{1/M_{n+1}})$, we can estimate $|z-\zeta_k^{(n)}|$ from below and get
\[ \Lambda_n(E) \le \max_{z\in \partial\,\left(E^{1/M_{n+1}}\right)} \sum_{k=0}^n M_{n+1}\, \frac{|w_{n+1}(z)|}{|\Delta_k^{(n)}|} .\]
Now a reasoning (using Taylor's formula) analogous to that from the proof of  \cite[Theorem 3.3]{Pl} yields
\[ \|p\|_{E^{1/M_{n+1}}} \le e\, \|p\|_E\]
for any polynomial $p$ of degree at most $n+1$. Consequently, we have
\begin{eqnarray*} \Lambda_n(E) \le M_{n+1}\, \sum_{k=0}^n \frac{\|w_{n+1}\|_{E^{1/M_{n+1}}}}{|\Delta_k^{(n)}|} &\le& M_{n+1}\, \sum_{k=0}^n \frac{e \, \|w_{n+1}\|_{E}}{|\Delta_k^{(n)}|}\\ & \le& M_{n+1}\, \frac{e \, (n+1)\, \|w_{n+1}\|_{E}}{\min\limits_{k\in\{0,...,n\}} |\Delta_k^{(n)}|}.\end{eqnarray*}

Next, fix again a $j\in \{0,...,n\}$ such that $|\Delta_j^{(n)}|=\min_k|\Delta_k^{(n)}|$ and observe that
\begin{eqnarray*}\left ( |\Delta_j^{(n)}|\right )^{n+1} \leq |\Delta_0^{(n)}| \cdot |\Delta_1^{(n)}|\cdot .... \cdot |\Delta_n^{(n)}| =V^2(\zeta^{(n)}) \leq V^2(\eta^{(n)}),
\end{eqnarray*}
where $\eta^{(n)}$ are the  Fekete points. Classical results of M. Fekete and G. Szeg\H o say that $ V(\eta^{(n)})^{\frac{2}{n(n+1)}} \longrightarrow {\rm cap} E$ as $n \to \infty$, hence by $(\ref{Deltamin})$ also $V(\zeta^{(n)})^{\frac{2}{n(n+1)}} \longrightarrow {\rm cap} E$ as $n \to \infty$.   By  \cite[Theorem 1.4 and Theorem 1.5]{BBCL}  this implies that $\lim_{n \to \infty} ||w_n||_E^{1/n}={\rm cap }\,E$.
Using  this fact again  we get
\[
\left(M_{n+1}\, \frac{e \, (n+1)\, \|w_{n+1}\|_{E}}{\min\limits_{k\in\{0,...,n\}} |\Delta_k^{(n)}|}\right)^{1/n}\longrightarrow 1, \quad n\rightarrow \infty.
\]
An obvious estimate $\Lambda_n(E)\ge 1$ completes the proof.
\end{proof}

\begin{rem}
  We already mentioned that (\ref{sublambda}) holds for extremal Fekete points.
 Condition (\ref{Deltamin}) is also easy to establish for these points, but the
connection between the two conditions has not been pointed out
before.
\end{rem}

 It is natural to ask about $\max_{k\in\{0,...,k\}}|\Delta_k^{(n)}|$.  For Fekete points, J. G\'orski proved (\cite{gorski})  using maximum principle for holomorphic functions that $$\max\limits_{k\in\{0,...,n\}} |\Delta^{(k)}(\eta^{(n)})|^{1/n} \longrightarrow {\rm cap}\,E \ \ \ \ as \ n\rightarrow\infty.$$ We will prove the same result for  pseudo Leja sequences. In what follows $a^{(n)}=\{a_0,...,a_n\}$.

\begin{prop}\label{prop:maxdelta}
  If $E$ is a compact regular set and $(a_n)_{n=0}^\infty$ is a pseudo Leja sequence of Edrei growth $(C_n)$ in $E$ then $$\max\limits_{k\in\{0,...,n\}} |\Delta^{(k)}(a^{(n)})|^{1/n} \longrightarrow {\rm cap}\,E \ \ \ \ as \ n\rightarrow\infty.$$
\end{prop}
\begin{proof}
Observe that $\max |\Delta^{(k)}(a^{(n)})| \ge |w_n(a_n)| \ge \|w_n\|_E/C_n$ and on the other hand $\Delta^{(k)}(a^{(n)})=w'_{n+1}(a_k)$ for all $k\in\{0,...,n\}, \ n\in\{1,2,...\}$. Therefore, by definition of $M_n$,  we have  $$ \frac{\|w_n\|_E}{C_n}\leq \max\limits_{k\in\{0,...,n\}}|\Delta^{(k)}(a^{(n)})|  \le M_{n+1}\|w_{n+1}\|_E.$$ Now it is sufficient to use $\|w_n\|_E^{1/n}\longrightarrow {\rm cap}\, E$, which  is a consequence of \cite[Theorem 2]{BCC} and \cite[Theorems 1.4 and 1.5]{BBCL}.
\end{proof}

\begin{rem}
  If $E$ is as above and  $A=\{\zeta^{(n)}: n\geq 1\}$ is an array  of distinct points  in $E$ satisfying Condition (\ref{Deltamin}),  then $$\max\limits_{k\in\{0,...,n\}} |\Delta^{(k)}(\zeta^{(n)})|^{1/n} \longrightarrow {\rm cap}\,E \ \ \ \ {\rm as} \ n\rightarrow\infty$$ as well. Indeed,
$\Delta^{(k)}(\zeta^{(n)})=w'_{n+1}(\zeta_k^{(n)})$ for all $k\in\{0,...,n\}, \ n\in\{1,2,...\}$. Hence $$\min\limits_{k\in\{0,...,n\}} |\Delta^{(k)}(\zeta^{(n)})|  \leq \max\limits_{k\in\{0,...,n\}} |\Delta^{(k)}(\zeta^{(n)})| \leq  M_{n+1}||w_{n+1}||_E$$ and the convergence  follows as in Proposition \ref{prop:maxdelta}.
\end{rem}

We have the following consequences of Theorems \ref{cor:Leb_dla_Leja} and  \ref{tw:rownowaznosc}:

\begin{cor}  \label{cor:deltaconvergence}  Assume that the outer boundary of $E$ is a finite union of  quasiconformal arcs and $(a_n)_{n=0}^\infty$ is a pseudo Leja sequence in $\partial\widehat{E}$ of bounded Edrei growth. Then
\begin{equation*}
  \lim\limits_{n\rightarrow\infty} \min\limits_{k\in\{0,...,n\}} |\Delta^{(k)}(a^{(n)})|^{1/n} = {\rm cap}\,E.
\end{equation*}
\end{cor}

\begin{cor} \label{cor:images}  Assume that the outer boundary of $E$ is a $\mathcal{C}^2$ curve and $a_n =\varphi(e_n), \ n\in \{0, 1,2,...\}$, where $\varphi$ denotes the conformal map of the exterior of the unit disk $\mathbb{D}$ onto the exterior of $E$ and $(e_n)_{n=0}^\infty$ is a  Leja sequence in $\mathbb{D}$.
Then
\begin{equation*}
  \lim\limits_{n\rightarrow\infty} \min\limits_{k\in\{0,...,n\}} |\Delta^{(k)}(a^{(n)})|^{1/n} = {\rm cap}\,E.
\end{equation*}
\end{cor}

For every $n \in \mathbb{N}$ let us now  fix a number $j=j_n \in \{0,...,n\}$ such that \begin{equation}\label{deltajn}
|\Delta_{j_n}^{(n)}|=\min_k|\Delta_k^{(n)}|                                                                                                         \end{equation}
  and form the Lagrange polynomial $L_n :=L^{(j_n)}(\cdot,\zeta^{(n)})$ (recall Notations \ref{nota:lagrange}). Under the assumption that $(\ref{Deltamin})$ holds, we have the following:
\begin{lem}\label{lem:normL_n} If $E$ is a polynomially convex regular compact subset of $\cj$, $(\ref{Deltamin})$ holds and $L_n$ is defined as above, then
 \[ \lim_{n\rightarrow \infty} \sqrt[n]{||L_n||_E}=1.\]
\end{lem}
\begin{proof}
As in the proof of Theorem \ref{tw:rownowaznosc} it is enough to observe that
\[
||L_n||_E\geq \frac1{{\rm diam}\, E}  \frac{\|w_{n+1}\|_E}{\min\limits_{k\in\{0,...,n\}} |\Delta_k^{(n)}|}
\]
and
\begin{eqnarray*}
||L_n||_E\leq \max_{z\in \partial\,\left(E^{1/M_{n+1}}\right)}  \frac{|w_{n+1}(z)|}{|z-\zeta_{j_n}^{(n)}|\, |\Delta_{j_n}^{(n)}|}\leq  M_{n+1}\, \frac{e  \|w_{n+1}\|_{E}}{\min\limits_{k\in\{0,...,n\}} |\Delta_k^{(n)}|}.
\end{eqnarray*}
\end{proof}

We can now prove an analogue of a theorem due to J. Siciak, \cite[Theorem 1]{Siciak}. We use symbols from  Notations \ref{nota:zbiory} and \ref{nota:lagrange}.
\begin{thm}\label{thm:bounds}
 Let $E$ be a polynomially convex  regular  compact subset of $\cj$ and $\omega$ be a modulus of continuity of $g_E$.  Let $A=\{\zeta^{(n)}: n\geq 1\}$ be an arbitrary array of distinct points  and let $L_n=L^{(j_n)}(\cdot,\zeta^{(n)})$ with $j_n$ chosen to satisfy $(\ref{deltajn})$. Then, for each $n$ and $z
 \in D_n$,
\begin{eqnarray*}
  \frac1n \log \frac1{||L_n||_E}\leq g_E(z)-\log \sqrt[n]{|L_n(z)|}\leq  \frac3n\log[(n+1)\Theta(E)]+\omega\left(\frac1{n^2}\right),
\\  \frac1n \log \frac1{||L_n||_E}\leq \log \frac{|\Delta_{j_n}^{(n)}|^{1/n}}{{\rm cap}\, E} \leq  \frac3n\log[(n+1)\Theta(E)]+\omega\left(\frac1{n^2}\right).
 \end{eqnarray*}
\end{thm}
\begin{proof} For the most part, especially concerning the upper bound, it is enough to follow the lines of the proof of \cite[Theorem 1]{Siciak} with $\zeta^{(n)}$ replacing the Fekete extremal points $\eta^{(n)}$.

Let $r(z):=\dist(z,E), R(z):=\max_{\zeta\in E}|z-\zeta|$. For any $k\in\{0,1,...,n\}$,
\begin{eqnarray*}
 |L_n(z)|=|L^{(k)}(z, \zeta^{(n)})|\frac{|\Delta^{(n)}_k|}{|\Delta^{(n)}_{j_n}|}\frac{|z-\zeta_k^{(n)}|}{|z-\zeta_{j_n}^{(n)}|}\geq |L^{(k)}(z, \zeta^{(n)})|\frac{r(z)}{R(z)},  \qquad z \in \mathbb{C}.
\end{eqnarray*}
The identity
$$
1\equiv \sum_{k=0}^n L^{(k)}(z, \zeta^{(n)})
$$
implies that
$$
  \max_{k\in\{0,1,...,n\}}|L^{(k)}(z, \zeta^{(n)})|\geq \frac1{n+1},\qquad z\in\cj.
$$
From the inequalities above we obtain therefore
$$
\frac{1}{|L_n(z)|}\leq (n+1)\frac{R(z)}{r(z)}, \qquad z\in D_\infty.
$$
On the other hand, by Definition \ref{defin:extremalfcn},
$$
\sqrt[n]{\frac{ \left|L_n(z)\right|}{||L_n||_E}}\leq \Phi_E(z), \qquad z\in\cj.
$$
These estimates  together with Lemma \ref{lemma:continuity} imply that the function \begin{eqnarray*}
U_n(z):=\left\{\begin{array}{ll}
  \displaystyle \log\frac{\Phi_E(z)}{|L_n(z)|^{1/n}}+\frac1n\log {||L_n||_E}, & z\in D_\infty \\
  \displaystyle \log\frac{{|\Delta_{j_n}^{(n)}|^{1/n}}}{{\rm cap}E}+\frac1n\log {||L_n||_E},& z=\infty,
\end{array}\right.
\end{eqnarray*}  is
 harmonic in $D_\infty\cup\{\infty\}$ and satisfies
 \[
 0\leq U_n(z)\leq \frac1n\log\frac{(n+1)R(z)}{r(z)}+\frac1n\log {||L_n||_E}+\omega(\delta),
 \]
 if $z\in D_\infty$ is such that $r(z)\leq \delta$.
  Therefore, by the maximum principle for harmonic functions, just as in the proof of \cite[Theorem 1(a)]{Siciak}, we obtain the assertion.
\end{proof}

 Note that the  lower bound   in the statement  of Theorem \ref{thm:bounds} is not necessarily $0$. However, if we assume additionally (\ref{Deltamin}), then by Lemma \ref{lem:normL_n}
the left-hand side tends to $0$ as $n \to \infty$.  This implies convergence of functions $\left(\log\sqrt[n]{|L_n|}\right)_{n=1}^\infty$ to $g_E$ in $D_\infty$.

We conclude this section by distinguishing some sequences of polynomials $(Q_n)$   such that  $g_E(z)-\log\sqrt[n]{|Q_n(z)|} \longrightarrow 0$  in $D_\infty$ as $n \to \infty$ with good  convergence rate. More precisely, we have the following:

\begin{prop}\label{prop: Bn}
Let $E$ be a regular compact subset of $\cj$ and $\omega$ be the modulus of continuity of its Green function $g_E$. Then there exist a sequence $(Q_n)_{n=1}^\infty$ of polynomials and a sequence $(B_n)_{n=1}^\infty$ of positive numbers satisfying $B_n =O(\frac{\log n}{n})$ such that, for all $n\in \nj$,  $Q_n$ is of degree $n$,  $\lim_{n\rightarrow \infty} \sqrt[n]{||Q_n||_E}=1$ and
$$\displaystyle g_E(z)-\log\sqrt[n]{|Q_n(z)|}\leq B_n+\omega(1/n^2), \qquad  z\in D_n.$$

\end{prop}

\begin{proof}
One can take a Fekete extremal $(n+1)$-tuple $\eta^{(n)}$ for $E$ ordered so that (\ref{order}) holds. Then the assertion follows for $Q_n:=L^{(0)}(\cdot, \eta^{(n)})$ from (\ref{zgory}) and (\ref{FeketeGreenbound}).

Alternatively, one can take a pseudo Leja sequence $(a_n)$ in $E$  satisfying condition (\ref{Deltamin}). Then the assertion follows for $Q_n:=L_n=L^{(j_n)}(\cdot, a^{(n)})$ from  Lemma \ref{lem:normL_n} and  Theorem \ref{thm:bounds}.
\end{proof}

\section{On approximation of compact planar sets by filled-in Julia sets}\label{s:filled}

 In this section we  revisit  the approximation of a planar set $E$ from above by filled-in Julia sets developed in \cite{arxiv}  focusing on  the rate of approximation.  Recall that \cite[Theorem 3.2]{arxiv} establishes a possibility of approximating a regular, polynomially convex $E$ by Julia sets of  polynomials satisfying $\lim\limits_{n\rightarrow\infty} \|w_n\|_E^{1/n} = {\rm cap}\,E$. As we already noted, this condition is implied by many statements. Here we will take polynomials with zeros satisfying  $\lim\limits_{n\rightarrow\infty} \Lambda_n(E,a^{(n)})^{1/n} = 1$ and explore  consequences of this assumption.
We will prove a version of the approximation result which keeps  track of approximation simultaneously in Hausdorff metric $\chi$ and in Klimek metric $\Gamma$ while also highlighting quantitative aspects of the process. Even more importantly, our version allows us to work with polynomials all of whose zeros except one lie on a level set of the Green function $g_E$.  Thanks to the results from Section \ref{s:lagrange} we have a wide choice of sequences of zeros  that are relatively easy to obtain (pseudo Leja points).

 Recall Notations \ref{nota:zbiory} and \ref{nota:lagrange} and  let us fix some final notations needed in the proof.
\begin{nota}\label{nota:stale}
Let $E$ be a nonempty compact set in $\cj$. We put $R(E):=\sup_{z\in E}|z|$.
Furthermore, if  $0\in  {\rm int}E$, define $r(E):=\dist(0, \partial E)$.
\end{nota}
Recall that $\D(a,R)$ stands for the closed disk with the center  $a \in \mathbb{C}$ and radius $R>0$. Note that if $0\in {\rm int}E$, then  $\D(0,r(E))$ is contained in $E$.

Let us now state our result.

\begin{thm}[\mbox{cf. \cite[Theorems 3.2 and 4.12]{arxiv}}] \label{thm: filled} Let $E$ be a regular compact subset of $\cj$. Fix $\varepsilon\in (0,1]$. Then there exist an arbitrarily small number $s>0$ and a polynomial $P$ (chosen for this $s$)  such that
\begin{equation}\label{inbetween}
 E\subset {\cal K}(P)\subset E_s\subset  \left(\widehat{E}\right)^\varepsilon.
\end{equation}In particular, $\Gamma(E, {\cal K}(P))\leq s$ and $\chi\left(\widehat E, {\cal K}(P)\right)\le\varepsilon.$
\end{thm}

\begin{proof}
Note first that if we fix $\varepsilon$, the existence of an $s>0$  such that
\begin{equation}\label{inclusions}
E_s\subset \left(\widehat E\right)^\varepsilon
\end{equation}
follows from  the fact that $\{E_s\}_{s>0}$ forms a neigbourhood base of $\widehat E$ (see \cite[Corollary 1]{pams}). Fix such an $s>0$. Note that one can take this $s$ arbitrarily small.  We have of course $E\subset \widehat E \subset E_s$.

\underline{Step I.} We will assume first that $0\in {\rm int}E$.  We will use the sequence of polynomials $(Q_n)$ and the sequence of constants $(B_n)$ from Proposition \ref{prop: Bn}.

Take $n\in \nj$ large enough to satisfy
\begin{equation}\label{s}
  \max\left(3B_n+3\omega\left(\frac1{n^2}\right),\frac6n\log||Q_n||_E, \frac6n\log\frac{R(E)+1}{r(E)}\right)\leq s.
\end{equation}
This is possible since the left hand side tends to zero by Proposition \ref{prop: Bn}. Fix such a large $n$. It follows that
\begin{eqnarray}
  \omega\left(\frac1{n^2}\right)&<&s; \label{1} \\
  B_n+\omega\left(\frac1{n^2}\right)&\leq&\frac s3; \label{2} \\
  \exp\left(\frac{ns}3\right)&>&\frac{R(E)+1}{r(E)} \label{3}; \\
 \exp\left(-\frac{ns}3\right)||Q_n||_E&\leq & \frac{r(E)}{R(E)+1}\label{4}.
\end{eqnarray}

Put $$P_n(z):=z\exp\left(-\frac{ns}3\right)Q_n(z), \qquad z\in\cj.$$

It follows from (\ref{1}) and the definition of $\omega$ that $s>\omega(1/n^2)\geq g_E(z)$ if $z\in  \left(\widehat E\right)^{1/n^2}$. This yields
$\Omega_s=\cj\setminus E_s\subset \cj\setminus  \left(\widehat E\right)^{1/n^2}=D_n$.

Fix now $\zeta\in \Omega_s$. Hence $\zeta\in D_n$ too.
It follows from the definition of $\Omega_s$, from Proposition \ref{prop: Bn} and  from (\ref{2}) that
$$
s< \log \sqrt[n]{\left|Q_n(\zeta)\right|} + \frac s3
$$
and hence
\begin{equation*}\label{zdolu}
|Q_n(\zeta)|>\exp\left(\frac{2ns}3\right).
\end{equation*}
Combining this with  (\ref{3}) gives
\begin{eqnarray}\label{odpychanie}
 \qquad |P_n(\zeta)|=|\zeta|\exp\left(-\frac{ns}3\right)|Q_n(\zeta)|&>& |\zeta|\exp\left(\frac{ns}3\right)> |\zeta| \frac{R(E)+1}{r(E)}. \end{eqnarray}
Note that $\displaystyle \frac{R(E)+1}{r(E)}>1$. Furthermore, it follows from the definition of $r(E)$ that $|P_n(\zeta)|\geq R(E)+1$ and therefore
$\dist\!\!\left(P_n(\zeta), \widehat{E}\right)\geq 1\geq \varepsilon$. Thus, in view of (\ref{inclusions})
\begin{equation}\label{powtarzanie} P_n(\zeta)\in \cj\setminus  \left(\widehat E\right)^\varepsilon\subset  \cj\setminus E_s=\Omega_s.\end{equation}

Since (\ref{odpychanie}) and (\ref{powtarzanie}) are true for every $\zeta\in \Omega_s$, we conclude that $\Omega_s\subset \cj\setminus {\cal K}(P_n)$, i.e. ${\cal K}(P_n)\subset E_s$.

On the other hand, if $z\in E$, it follows from definition of $R(E)$ and (\ref{4}) that
$$
|P_n(z)|\leq R(E)\exp\left(-\frac{ns}3\right)||Q_n||_E\leq  R(E)\frac{r(E)}{R(E)+1}<r(E).
$$
Therefore, $P_n(E)\subset \mathbb{D}(0,r(E))\subset E$ (the latter inclusion is a consequence of the definition of $r(E)$) and this means that $E\subset {\cal K}(P_n)$.
\\

\underline{Step II.} We will assume now that $0\in E$. Then $0\in {\rm int}E_{s/2}$ because $g_E$ is continuous. Recall that by Proposition \ref{prop:sublevel} each $E_s$ is polynomially convex and moreover $E_s=(E_{s/2})_{s/2}$. In view of  Step I with $E$ replaced by $E_{s/2}$ and $s$ replaced by $s/2$ there exist a polynomial $P$ such that
$$
E\subset E_{s/2}\subset {\cal K}(P)\subset E_{s}.
$$

\underline{Step III.} Let $E$ be any regular compact subset of $\cj$ and $w\in E$. We consider the translation of $E$ via vector $-w$, namely $E-w:=\{z\in \cj: z+w\in E\}$. Then $0\in E-w$ and the conclusion follows from Step II and the properties of the Green function (namely from the equality $g_{E-w}(z)=g_E(z+w)$).
\end{proof}

\begin{rem}
A  careful analysis of Step II shows that the zeros of the polynomial $P$, with one exception, do not have to be contained in $E$. Namely, we may use the Lagrange polynomial with nodes in the sublevel set $E_{s'}$ with $s'<s$. These points can be contained in the boundary of $E_{s'}$, while $E\subset {\rm int}E_{s'}$. Considering nodes in $E_{s'}$ offers more flexibility: by Proposition \ref{prop:C2arcs}, the boundary of each $E_{s'}$  is a finite union of quasiconformal arcs, so  by Corollary \ref{cor:deltaconvergence}, every pseudo Leja sequence in $\partial E_{s'}$ of bounded Edrei growth satisfies (\ref{Deltamin}).
 Therefore we can approximate $E$ by filled-in Julia sets defined with use of pseudo Leja points for $E_{s'}$ (cf. the second part of the proof of Proposition \ref{prop: Bn}).
\end{rem}

 In the proof of  our Theorem \ref{thm: filled} one can clearly see what affects  the rate of approximation of a compact, polynomially convex, regular $E$ by polynomial filled-in Julia sets:  some basic geometric quantities related to $E$, the modulus of continuity of the Green function $g_E$ as well as the rate of convergence of $\frac1n\log||Q_n||_E$ to $0$, where $Q_n$ are polynomials satisfying Proposition \ref{prop: Bn}.  In our construction $Q_n$ are of the form $L^{(j_n)}(\cdot, \zeta^{(n)})$, which is important for the estimates below. To get  more explicit relations, one can use the estimates on $||Q_n||$ from the proof of Lemma \ref{lem:normL_n}. Alternatively, one can apply the inequalities
\[
\left\|L^{(j_n)}(\cdot, \zeta^{(n)})\right\|_E\leq \frac{\Delta_n}{|\Delta_{j_n}^{(n)}|}, \qquad n \in \mathbb{N}.
\]
Here $\Delta_n:=\sup_{\xi^{(n)}\subset E}\left(\min_{ k\in\{0,1,...,n\}} |\Delta^{(k)}(\xi^{(n)})|\right)$ depends only on $E$.  Moreover, $\lim_{n\rightarrow \infty} \sqrt[n]{\Delta_n}={\rm cap} E$
(for the proof see \cite[p. 258]{Lejabook}). Finally, one can use the general inequalities $\|L^{(j_n)}(\cdot, \zeta^{(n)})\| \leq \Lambda_n(E, \zeta^{(n)})$ (see Definition \ref{def:Lebesgue}). But one still needs more detailed  information about convergence of the mentioned quantities to make conclusions about the rate of approximation of sets.

\section{An approach to estimation of the rate of convergence}\label{s:rates}

 For sets $E$ as above  with $g_E$ which is H\"older continuous, \cite[Theorem 4.12]{arxiv}  gives an estimate of the rate of approximation of $E$ in the metric $\Gamma$ by polynomial filled-in Julia sets. Namely, for $n \in \mathbb{N}$, let $S_n(E)$  be the infimum of $s > 0$ for which there exists a polynomial $P_n$ of
degree $n$ such that (\ref{inbetween}) holds. Then there exists a real number $c = c(E)>0$ such that $S_n(E) \leq c\frac{\log n}{\sqrt{n}}$. This result is proved with the use of \cite[Theorem 2.2]{Pritsker11}, which was  shown under the assumption that  the zeros $\zeta_0^{(n-1)},...,\zeta_{n-1}^{(n-1)}$ of polynomials $w_n$  which  approximate $g_E$ in $\mathbb{C}\setminus E$  satisfy the condition $\log  V(\zeta^{(n)})
-{\rm cap}E \leq C_1\frac{\log n}{n}$  (see Notation \ref{nota:lagrange}).  So far, the only systems of points known to satisfy this condition are extremal Fekete arrays and Leja sequences (or their small perturbations). Both Fekete and Leja points were considered in \cite{arxiv} to derive the rate $O(\frac{\log n}{\sqrt{n}})$ of approximation by Julia sets. We will begin this section by giving two examples of families of points which improve this rate to  $O(\frac{\log n}{n})$ or, in some particular situation, to $O(1/n^{2\alpha})$. In the remaining part we assume that $E$ is a compact, polynomially convex subset of $\mathbb{C}$ and that the Green function $g_E$ is H\"older continuous, that is, for all $\delta >0$ one has $\omega(\delta)=A\delta^\alpha$ with suitable constants $A, \alpha$. We have the following corollary from Theorem \ref{thm: filled}:\\

\begin{cor}\label{wn:Gamma}
  Let $E$ be a compact subset of $\cj$ such that $g_E$ is H\"older continuous with the constant $A$ and the exponent $\alpha\in (0,1]$. There exist for each $n\in\nj$, large enough,  a polynomial $P_n$ of degree $n+1$ such that
  \begin{itemize}
    \item $\Gamma(E, {\cal K}(P_n))\leq (3A+12) n^{-1}{\log(n+1)}, \quad {\it if} \quad \alpha\in [1/2, 1]$;
    \item $\Gamma(E, {\cal K}(P_n))\leq (3A+12) n^{-2\alpha}, \quad {\it if} \quad \alpha\in (0,1/2)$.
  \end{itemize}
\end{cor}
\begin{proof}
 We  take  the Fekete points $\eta^{(n)}$ as the nodes of the Lagrange polynomials which we will use in the construction (cf. the proof of Proposition \ref{prop: Bn}). We will refer to the steps of the proof of Theorem \ref{thm: filled}.

\underline{Step I.}  Observe that in this case  the norm of the Lagrange polynomial is not larger than 1  (cf. (\ref{zgory})). Moreover, if we  assume $n+1>(R(E)+1)/r(E)$ then formula (\ref{s}) will reduce to
\begin{equation}\label{scor}
  \frac9n\log[(n+1)\Theta(E)]+3\omega\left(\frac1{n^2}\right)\leq s.
\end{equation}
Hence
if we fix first $s$, then choose $n$ large enough and satisfying (\ref{scor}) and fix it too, we have
\begin{eqnarray*}
 s\geq\frac{3A}{n^{2\alpha}}+\frac9n\log(n+1)+\frac3n\log({\rm diam}E+2).
\end{eqnarray*}
Therefore if $\alpha\in [1/2,1]$, then we can take $s_n=(3A+12)n^{-1}\log(n+1)$ (resp. if $\alpha\in(0,1/2)$, then $s_n=(3A+12)n^{-2\alpha}$) for large $n$.

Step II. The estimates follow from Step I thanks to Lemma \ref{lem:holder}.

Step III. The same is true for the image of the set $E$ under a translation.
\end{proof}

Let us see an example of a polynomial $P_n$ satisfying the assertions of Corollary \ref{wn:Gamma}.
 \begin{exa}\label{exa:Fekete}
 Without loss of generality we can assume that $0\in E$. Let $A, \alpha$ be as in the assumptions of Corollary  \ref{wn:Gamma} and let $$s_n=\frac{3A}{n^{2\alpha}}+\frac9n\log(n+1)+\frac3n\log({\rm diam}E+3).$$ If $0\in {\rm int} E$ let
 $\eta^{(n)}$ be a Fekete $(n+1)$-tuple in $E$ and if $0\in E\setminus {\rm int}E$, let it be a Fekete $(n+1)$-tuple for $E_{\tau_n}$ where
$\tau_n=(3/n)(\log(n+1)-\log({\rm diam E}+3))$. It should be
 ordered so that (\ref{order}) holds. Then for $n$ large enough and $z\in \cj$
 $$ P_n(z)=z\exp\left(-\frac{ns_n}3\right)L^{(0)}(z, \eta^{(n)}).$$
  \end{exa}

  \bigskip

The following corollary is a straightforward consequence of Theorem \ref{thm: filled} (in a similar way as Corollary \ref{wn:Gamma}) and considers the most general situation concerning polynomials which we can use.

\begin{cor}\label{cor:general}
 Let $E$ be a polynomially convex compact subset of $\cj$ such that $0\in {\rm int}E$ and $g_E$ is H\"older continuous with the constant $A$ and the exponent $\alpha\in (0,1]$. Let $(B_n)_{n=1}^\infty$ and $(Q_n)_{n=1}^\infty$ be such as in Proposition $\ref{prop: Bn}$. Define
  \begin{eqnarray*}
    t_n&:=&\frac{3A}{n^{2\alpha}}+\frac9n\log(n+1)+\frac3n\log({\rm diam}E+2);\\
     s_n &:=& \max\left(t_n,\frac6n\log||Q_n||_E, \frac6n\log\frac{R(E)+1}{r(E)}\right); \\
    P_n(z)&:=& z\exp\left(-\frac{ns_n}3\right)Q_n(z), \qquad \qquad \qquad z\in \cj.
  \end{eqnarray*}
  Then $\Gamma(E, {\cal K}(P_n))\leq s_n$.
\end{cor}

 Let us consider now a special case in which we use a pseudo Leja sequence.

\begin{exa}\label{exa:C2curve}
Let the outer boundary of $E$ be a $\mathcal{C}^2$ curve. We know $g_E$ is H\"older continuous, say with constants $A$ and $\alpha$. We assume that $0\in {\rm int}E$ once again. Let $a_n :=\varphi(e_n), \ n\in \{0, 1,2,...\}$,  where $\varphi$ denotes the conformal map of the exterior of the unit disk $\mathbb{D}$ onto the exterior of $E$ and $(e_n)_{n=0}^\infty$ is a  Leja sequence in $\mathbb{D}$. The points $a_n$ form a pseudo Leja sequence (not necessarily of bounded Edrei growth) and by Corollary \ref{cor:images} they satisfy (\ref{Deltamin}).
We can take in this case $Q_n:=L^{(j_n)}(\cdot, a^{(n)})$, where
$j_n$ is chosen to satisfy (\ref{deltajn}). It follows from Corollary \ref{cor:general} (with the notations taken from there) that
$\Gamma(E, {\cal K}(P_n))\leq s_n$. Moreover, in this case $s_n=O(\frac{\log n}{n})$, since $||Q_n||_E\leq \Lambda_n(E)$ and one has $\Lambda_n(E)=O(n^{1+2\beta/\log 2})$ with $\beta>0$ depending only on $E$ by \cite[Theorem 1.3]{Ir}.
\end{exa}

Note that while Example \ref{exa:Fekete} improves an already existing estimate, Example \ref{exa:C2curve} gives an  explicit estimate in a case that was not previously covered, because  pseudo Leja points are not mentioned either in \cite{arxiv} or in \cite{Pritsker11}.

Explicit estimations for the Hausdorff metric are not discussed in \cite{arxiv}. We can find such an estimate in the special case of sets whose Green functions satisfy a special type of the \L ojasiewicz inequality, introduced by M. T. Belghiti and L. P. Gendre in their Ph.D. dissertations (cf. \cite{LS} and the references given there). Let us recall the definition (or actually an equivalent condition, cf. \cite{wielwyp}). We say that a compact regular set $E$ satisfies the {\it \L ojasiewicz-Siciak condition} if there exist positive constants $B, \beta$ such that for $U$ bounded neighbourhood of $E$
\begin{equation}\label{LS}
g_E(z)\geq B(\dist(z, E))^\beta,\qquad z\in U.
\end{equation}
 Every compact subset of $\rj$ satisfies \L ojasiewicz-Siciak condition (see \cite{BCE}). Moreover, it follows that every compact regular subset of any line in the complex plane satisfies it too, since $g_{f^{-1}(E)}=g_E\circ f$ for any compact set $E$ and any linear function $f:\cj\ni z\mapsto az+b\in \cj$. Another positive example is a star-like set $E:=\left\{z=r\exp(2\pi i k/n): r\in [0,1], k\in\{1,...,n\}\right\}$, $n\in\nj$, (see \cite{BCE}).
 In the negative direction, it is worth mentioning that the union of two tangent disks in $\mathbb{C}$ does not satisfy the \L ojasiewicz-Siciak condition. This was shown  by J. Siciak (see \cite{cosc}).

The following lemma shows how the \L ojasiewicz-Siciak condition grants a relation between the Klimek metric and the Hausdorff one in our case.
\begin{lem}\label{lem:LS} Assume that $E, F$ are regular compact planar sets and $U$ is an open set such that $E\subset F\subset U$. Let $E$ satisfy the \L ojasiewicz-Siciak condition.  Then
$$
\chi(E,F)\leq \left(\frac1B\Gamma(E,F)\right)^{1/\beta},
$$
where $B, \beta$ are the constants from $(\ref{LS})$ for $E$ and $U$.
\end{lem}
\begin{proof} Inclusion $E\subset F$ yields $$
\Gamma(E, F)=||g_E||_{F},\qquad \chi(E, F) =||\dist(\cdot, E)||_{F}.
$$
The assertion follows from (\ref{LS}).
\end{proof}

  A big family of examples of sets satisfying  the \L ojasiewicz-Siciak condition is the Lithner family (considered in \cite{lith}) of totally disconnected uniformly perfect sets, which can be called the generalized Cantor sets.  The \L ojasiewicz-Siciak condition for them was established in   \cite{cosc}. Being uniformly perfect, these sets also have H\"older continuous Green functions.
  Let us look at another family of sets enjoying both properties, one consisting of polynomial
preimages of the unit disk, i.e. sets of the form $p^{-1}(\mathbb{D}(0,1))=\{z\in\cj: |p(z)|\leq 1\}$, where $p$ is a non-constant polynomial. Since
  $$g_{p^{-1}(\mathbb{D}(0,1))}=\frac1{{\rm deg}p}g_{\mathbb{D}(0,1)}\circ p=\frac1{{\rm deg}p}\max\left(0, \log|p|\right)$$
(for the first equality see e.g. \cite[p.134]{Ransford} or \cite[Theorem 5.3.1]{Kksiazka}), these sets satisfy the \L ojasiewicz-Siciak condition and have H\"older continuous Green functions (cf. \cite{wielwyp}). Furthermore, any simple polygon and any simple (not closed) polygonal chain on the plane  (as well as a union of any finite family of pairwise disjoint sets of this kind) have also both properties. The \L ojasiewicz-Siciak condition can be deduced e.g. from the fact that the complement is a H\"older domain (see \cite{BCE} for details) or from
  \cite[Theorem 1.3]{Pi14}.
  For more information regarding the (sometimes subtle) relations between the \L ojasiewicz-Siciak condition and the H\"older continuity property of Green functions we refer to \cite{Pi14}, \cite{wielwyp}.

We are ready to state our last result concerning a rate of convergence in the Hausdorff metric.

\begin{cor}\label{wn:Loja}
Let $E$ be a compact set
satisfying \L ojasiewicz-Siciak condition. Assume also that $g_E$ is H\"older continuous with exponent $\alpha\in (0,1]$. Then there exist positive constants $C, \kappa$, depending only on the set $E$, and for each $n\in\nj$, large enough,  a polynomial $P_n$ of degree $n+1$ such that
  \begin{itemize}
    \item $\chi(E, {\cal K}(P_n))\leq C \left(n^{-1}\log(n+1)\right)^\kappa, \quad {\it if} \quad \alpha\in [1/2, 1]$;
    \item $\chi(E, {\cal K}(P_n))\leq C n^{-2\alpha\kappa}, \quad {\it if} \quad \alpha\in (0,1/2)$.
  \end{itemize}
\end{cor}

\begin{proof}  Recall that since $E$ satisfies \L ojasiewicz-Siciak condition, it is polynomially convex (see  \cite{wielwyp}).  Fix $P_n$ from Corollary \ref{wn:Gamma}. We know that $E\subset {\cal K}(P_n)\subset E^1$ by Theorem \ref{thm: filled}. Let $B, \beta$ be the constants from (\ref{LS}) for $E$ and $U:=E^1$. Put $\kappa:=1/\beta$. The assertion follows from Lemma \ref{lem:LS} (with $F:=  {\cal K}(P_n)$)  and Corollary \ref{wn:Gamma}.
\end{proof}

\subsection*{Acknowledgements}
 This paper was written when the third author was visiting the Institute of Mathematics of Jagiellonian University (Krak\'ow, Poland)  on her study leave from American Mathematical Society. During that time the second and the third author also visited the Department of Mathematics at Uppsala University (Sweden).   Thanks are extended to both institutions for their hospitality and particularly to Maciej Klimek for the invitation to Uppsala and helpful conversations.  We also thank Malik Younsi for discussions about the contents of older and newer versions of \cite{arxiv}. And last but not least, the referee deserves thanks for careful reading
and many useful comments. 

The research of the first and the second author was partially supported by the NCN Grant No. 2013/11/B/ST1/03693.

\end{document}